\documentclass[12pt]{amsart}
\usepackage{amsmath, amsthm, amscd, amsfonts, amssymb, graphicx, color,float,pgf,tikz}
\usepackage{amssymb,fontenc}
\usepackage{latexsym,wasysym,mathrsfs}
\usepackage{hyperref}
\usetikzlibrary{arrows}
\usepackage[square,numbers,sort&compress]{natbib}

\makeatletter \oddsidemargin.9375in \evensidemargin \oddsidemargin
\newtheorem{theorem}{Theorem}[section]
\newtheorem{lemma}[theorem]{Lemma}
\newtheorem{proposition}[theorem]{Proposition}

\theoremstyle{definition}
\newtheorem{definition}[theorem]{Definition}
\newtheorem{example}[theorem]{Example}

\newtheorem{remark}[theorem]{Remark}
\numberwithin{equation}{section}

\newcommand{\be}{\begin{equation}}
\newcommand{\ee}{\end{equation}}
\newcommand{\bee}{\begin{example}}
\newcommand{\eee}{\end{example}}

\numberwithin{equation}{section}

\usepackage{etoolbox}
\makeatletter
\patchcmd{\@settitle}{\uppercasenonmath\@title}{}{}{}
\patchcmd{\@setauthors}{\MakeUppercase}{}{}{}
\makeatother

\allowdisplaybreaks
\begin{document}

\title{Continuous Controlled K-Frame for Hilbert $C^{\ast}$-Modules}

\author{Hamid Faraj$^{1*}$, Samir Kabbaj$^1$, Hatim Labrigui$^1$, Abdeslam Touri$^1$  and Mohamed Rossafi$^2$}

\address{$^{1}$Laboratory of Partial Differential Equations, Spectral Algebra and Geometry Department of Mathematics, Faculty of Sciences, University Ibn Tofail, Kenitra, Morocco}

\email{\textcolor[rgb]{0.00,0.00,0.84}{farajham19@gmail.com}}

\email{\textcolor[rgb]{0.00,0.00,0.84}{samkabbaj@yahoo.fr}}

\email{\textcolor[rgb]{0.00,0.00,0.84}{hlabrigui75@gmail.com}}

\email{\textcolor[rgb]{0.00,0.00,0.84}{touri.abdo68@gmail.com}}

\address{$^{3}$LaSMA Laboratory Department of Mathematics Faculty of Sciences, Dhar El Mahraz University Sidi Mohamed Ben Abdellah, B. P. 1796 Fes Atlas, Morocco}
\email{\textcolor[rgb]{0.00,0.00,0.84}{rossafimohamed@gmail.com; mohamed.rossafi@usmba.ac.ma}}



\date{%10/03/2020; %Accepted: zzzzzz.
	\newline \indent $^{*}$Corresponding author}

\begin{abstract}
	In this paper, we introduce and we study the concept of Continuous Controlled K-Frame for Hilbert $C^{\ast}$-Modules wich are generalizations of discrete Controlled K-Frames.
\end{abstract}
\maketitle
\vspace{0.1in}
\section{\textbf{Introduction and preliminaries}}
The concept of frames in Hilbert spaces has been introduced by
Duffin and Schaeffer \cite{Duf} in 1952 to study some deep problems in nonharmonic Fourier series. After the fundamental paper \cite{13} by Daubechies, Grossman and Meyer, frame theory began to be widely used, particularly in the more specialized context of wavelet frames and Gabor frames \cite{Gab}. Frames have been used in signal processing, image processing, data compression and sampling theory. 
The concept of a generalization of frames to a family indexed by some locally compact space endowed with a Radon measure was proposed by G. Kaiser \cite{15} and independently by Ali, Antoine and Gazeau \cite{11}. These frames are known as continuous frames. Gabardo and Han in \cite{14} called these frames associated with measurable spaces, Askari-Hemmat, Dehghan and Radjabalipour in \cite{12} called them generalized frames and in mathematical physics they are referred to as coherent states \cite{11}. 
In 2012, L. Gavruta \cite{02} introduced the notion of K-frames in Hilbert space to study the atomic systems with respect to a bounded linear operator K. Controlled frames in Hilbert spaces have been introduced by P. Balazs \cite{01} to improve the numerical efficiency of iterative algorithms for inverting the frame operator. Rahimi \cite{05} defined the concept of controlled K-frames in Hilbert spaces and showed that controlled K-frames are equivalent to K-frames due to which the controlled operator C  can be used as preconditions in applications.
Controlled frames in $C^{\ast}$-modules were introduced by Rashidi and Rahimi \cite{03}, and the authors showed that they share many useful properties with their corresponding notions in a Hilbert space. We extended the results of frames in Hilbert spaces to  Hilbert $C^{\ast}$-modules (see \cite{R1}, \cite{R2}, \cite{R3}, \cite{R4}, \cite{R5}, \cite{R6}, \cite{R7}, \cite{R8}, \cite{R9}, \cite{R10}, \cite{R11}, \cite{R12})\\
Motivated by the above literature, we introduce the notion of a continuous controlled K-frame in Hilbert $C^{\ast}$-modules.

In the following we briefly recall the definitions and basic properties of $C^{\ast}$-algebra, Hilbert $\mathcal{A}$-modules. Our references for $C^{\ast}$-algebras as \cite{{Dav},{Con}}. For a $C^{\ast}$-algebra $\mathcal{A}$ if $a\in\mathcal{A}$ is positive we write $a\geq 0$ and $\mathcal{A}^{+}$ denotes the set of positive elements of $\mathcal{A}$.
\begin{definition}\cite{Pas}	
	Let $ \mathcal{A} $ be a unital $C^{\ast}$-algebra and $\mathcal{H}$ be a left $ \mathcal{A} $-module, such that the linear structures of $\mathcal{A}$ and $ \mathcal{H} $ are compatible. $\mathcal{H}$ is a pre-Hilbert $\mathcal{A}$-module if $\mathcal{H}$ is equipped with an $\mathcal{A}$-valued inner product $\langle.,.\rangle_{\mathcal{A}} :\mathcal{H}\times\mathcal{H}\rightarrow\mathcal{A}$, such that is sesquilinear, positive definite and respects the module action. In the other words,
	\begin{itemize}
		\item [(i)] $ \langle x,x\rangle_{\mathcal{A}}\geq0 $ for all $ x\in\mathcal{H} $ and $ \langle x,x\rangle_{\mathcal{A}}=0$ if and only if $x=0$.
		\item [(ii)] $\langle ax+y,z\rangle_{\mathcal{A}}=a\langle x,z\rangle_{\mathcal{A}}+\langle y,z\rangle_{\mathcal{A}}$ for all $a\in\mathcal{A}$ and $x,y,z\in\mathcal{H}$.
		\item[(iii)] $ \langle x,y\rangle_{\mathcal{A}}=\langle y,x\rangle_{\mathcal{A}}^{\ast} $ for all $x,y\in\mathcal{H}$.
	\end{itemize}	 
	For $x\in\mathcal{H}, $ we define $||x||=||\langle x,x\rangle_{\mathcal{A}}||^{\frac{1}{2}}$. If $\mathcal{H}$ is complete with $||.||$, it is called a Hilbert $\mathcal{A}$-module or a Hilbert $C^{\ast}$-module over $\mathcal{A}$. For every $a$ in $C^{\ast}$-algebra $\mathcal{A}$, we have $|a|=(a^{\ast}a)^{\frac{1}{2}}$ and the $\mathcal{A}$-valued norm on $\mathcal{H}$ is defined by $|x|=\langle x, x\rangle_{\mathcal{A}}^{\frac{1}{2}}$ for $x\in\mathcal{H}$.
	
	Let $\mathcal{H}$ and $\mathcal{K}$ be two Hilbert $\mathcal{A}$-modules, A map $T:\mathcal{H}\rightarrow\mathcal{K}$ is said to be adjointable if there exists a map $T^{\ast}:\mathcal{K}\rightarrow\mathcal{H}$ such that $\langle Tx,y\rangle_{\mathcal{A}}=\langle x,T^{\ast}y\rangle_{\mathcal{A}}$ for all $x\in\mathcal{H}$ and $y\in\mathcal{K}$.
	
We reserve the notation $End_{\mathcal{A}}^{\ast}(\mathcal{H},\mathcal{K})$ for the set of all adjointable operators from $\mathcal{H}$ to $\mathcal{K}$ and $End_{\mathcal{A}}^{\ast}(\mathcal{H},\mathcal{H})$ is abbreviated to $End_{\mathcal{A}}^{\ast}(\mathcal{H})$.
\end{definition}  
\begin{lemma} \label{1} \cite{Ara}.
	Let $\mathcal{H}$ and $\mathcal{K}$ two Hilbert $\mathcal{A}$-modules and $T\in End_{\mathcal{A}}^{\ast}(\mathcal{H})$. Then the following statements are equivalente:
	\begin{itemize}
		\item [(i)] $T$ is surjective.
		\item [(ii)] $T^{\ast}$ is bounded below with respect to norm, i.e, there is $m>0$ such that $\|T^{\ast}x\|\geq m\|x\|$, $x\in\mathcal{K}$.
		\item [(iii)] $T^{\ast}$ is bounded below with respect to the inner product, i.e, there is $m'>0$ such that, $$\langle T^{\ast}x,T^{\ast}x\rangle_\mathcal{A}\geq m'\langle x,x\rangle_\mathcal{A} , x\in\mathcal{K}$$
	\end{itemize}
\end{lemma}

\begin{lemma} \label{2}\cite{Pas} 
	Let $\mathcal{H}$ and $\mathcal{K}$ two Hilbert $\mathcal{A}$-modules and $T\in End_{\mathcal{A}}^{\ast}(\mathcal{H})$. Then the following statements are equivalente,
	\begin{itemize}
		\item [(i)] The operator T is bounded and $\mathcal{A}$-linear.
		\item  [(ii)]	There exist $0\leq k$ such that
		\begin{equation*}
		\langle Tx, Tx \rangle_\mathcal{A}\leq  k \langle x, x \rangle_\mathcal{A}  \qquad x\in \mathcal{H}.
		\end{equation*}
		
	\end{itemize}	
\end{lemma}
For the following theorem, R(T) denote the range of the operator T.
\begin{theorem} \label{3} \cite{Zha}
	Let $\mathcal{H}$ be a Hilbert $\mathcal{A}$-module over a $C^{\ast}$-algebra $\mathcal{A}$ and let $T, S$ two operators for $End_{\mathcal{A}}^{\ast}(\mathcal{H})$. If $R(S)$ is closed, then the following statements are equivalent:
	\begin{itemize}
		\item [(i)] $R(T)\subset R(S)$.
		\item [(ii)] $ TT^{\ast}\leq \lambda^{2} SS^{\ast}$ for some $\lambda\geq0$.
		\item [(iii)] There exists $Q\in End_{\mathcal{A}}^{\ast}(\mathcal{H})$ such that $T=SQ$.
	\end{itemize}
\end{theorem}
\section{\textbf{Continuous Controlled K-Frame for Hilbert $C^{\ast}$-Modules}}
Let $X$ be a Banach space, $(\Omega,\mu)$ a measure space, and  $f:\Omega\to X$ a measurable function. Integral of the Banach-valued function $f$ has been defined by Bochner and others. Most properties of this integral are similar to those of the integral of real-valued functions. Since every $C^{\ast}$-algebra and Hilbert $C^{\ast}$-module is a Banach space thus we can use this integral and its properties.

Let $\mathcal{H}$ and $\mathcal{K}$ be two Hilbert $C^{\ast}$-modules, $\{\mathcal{K}_{w}: w\in\Omega\}$  is a family of subspaces of $\mathcal{K}$, and $End_{\mathcal{A}}^{\ast}(\mathcal{H},\mathcal{K}_{w})$ is the collection of all adjointable $\mathcal{A}$-linear maps from $\mathcal{H}$ into $\mathcal{K}_{w}$.
We define
\begin{equation*}
	\oplus_{w\in\Omega}\mathcal{K}_{w}=\{x=\{x_{w}\}_{w\in\Omega}: x_{w}\in \mathcal{K}_{w}, \int_{\Omega}\|x_{w}\|^{2}d\mu(w)<\infty\}.
\end{equation*}
For any $x=\{x_{w}: w\in\Omega\}$ and $y=\{y_{w}: w\in\Omega\}$, if the $\mathcal{A}$-valued inner product is defined by $\langle x,y\rangle_\mathcal{A}=\int_{\Omega}\langle x_{w},y_{w}\rangle_\mathcal{A} d\mu(w)$, the norm is defined by $\|x\|=\|\langle x,x\rangle_\mathcal{A}\|^{\frac{1}{2}}$. Therefore, $\oplus_{w\in\Omega}\mathcal{K}_{w}$ is a Hilbert $C^{\ast}$-module(see [14]).\\
Let $\mathcal{A}$ be a $C^{\ast}$-algebra, $l^2({\mathcal{A}})$ is defined by, 
$$l^2({\mathcal{A}})=\{\{a_\omega\}_{w\in\Omega} \subseteq \mathcal{A} :\|\int_{\Omega}a_{\omega} a_{\omega}^{\ast} d\mu({\omega})\| <\infty\}.$$
$l^2({\mathcal{A}})$ is a Hilbert $C^{\ast}$-module (Hilbert $\mathcal{A}-module$) with pointwise operations and the inner product defined as, 
$$\langle \{a_\omega\}_{w\in\Omega} , \{b_\omega\}_{w\in\Omega}\rangle_\mathcal{A} =\int_{\Omega}a_{\omega} b_{\omega}^{\ast} d\mu({\omega}), 
\{a_\omega\}_{w\in\Omega}, \{b_\omega\}_{w\in\Omega} \in l^2({\mathcal{A}}),$$
and, $$\|\{a_\omega\}_{w\in\Omega}\|=(\int_{\Omega}a_{\omega} a_{\omega}^{\ast} d\mu({\omega}))^{\frac{1}{2}}.$$

\begin{definition} \label{17}                                      
	Let $\mathcal{H}$ be a Hilbert $\mathcal{A}$-module over a unital $C^{\ast}$-algebra, and 
	$K\in End_{\mathcal{A}}^{\ast}(\mathcal{H})$. A mapping F: $\Omega\rightarrow \mathcal{H}$ is called a continuous  K-Frame for $\mathcal{H}$ if :  
	\begin{itemize}
		\item F is weakly-measurable, ie,  for any $f\in \mathcal{H}$, the map
		\\
		$ w\rightarrow \langle f,F(w)\rangle_{\mathcal{A}}$ is  measurable on ${\Omega}$.\\
		\item There exist two strictly positive constants  $A$ and $B$  such that
		\begin{equation} \label{127}
		A\langle  K^{\ast}f, K^{\ast}f\rangle_\mathcal{A} \leq\int_{\Omega}\langle f,F(w)\rangle_\mathcal{A} \langle F(w),f\rangle_\mathcal{A} d{\mu(w)}\leq B\langle f,f\rangle_\mathcal{A} , f\in \mathcal{H}.
		\end{equation}
	\end{itemize}
	The elements $A$ and $B$ are called continuous  K-frame bounds. 
	
	If $A=B$ we call this Continuous  K-Frame a continuous tight  K-Frame, and if $A=B=1$ it is called a continuous Parseval  K-Frame. If only the right-hand inequality of \eqref{127} is satisfied, we call  F a continuous  bessel mapping with Bessel bound $B$.\\
	Let F be a  continuous  bessel mapping for Hilbert $C^{\ast}$- module $\mathcal{H}$ over $\mathcal{A}$.\\
	The operator  T :$ \mathcal{H} \rightarrow l^2(\mathcal{A})$ defined by, $$Tf=\{\langle f, F(\omega)\rangle _\mathcal{A}\}_{\omega \in \Omega},$$
	is called the analysis operator.\\
	There adjoint operator $T^{\ast} : l^2(\mathcal{A})  \rightarrow \mathcal{H}$ given by, $$T^{\ast}(\{a_\omega\}_{\omega \in \Omega})=\int_{\Omega} a_\omega F(\omega) d\mu(\omega),$$
	is called the synthesis operator.\\ By composing T and $T^{\ast}$, we obtain the continuous K-frame operator, $S : \mathcal{H}  \rightarrow \mathcal{H}$ defined by 
	$$Sf=\int_{\Omega}\langle f, F(\omega)\rangle_\mathcal{A} F(\omega) d\mu(\omega).$$
    It's clear to see that $S$ is positive, bounded and selfadjoint (see \cite{11}).

\end{definition}
   
    For the following definition we need to introduce, ${GL^{+}(\mathcal{H})}$ be the set of all positive bounded linear invertible operators on $\mathcal{H}$ with bounded inverse.

\begin{definition} \label{6}                                           
	Let  $\mathcal{H}$ be a Hilbert $\mathcal{A}$-module over a unital $C^{\ast}$-algebra and 
	$K\in End_{\mathcal{A}}^{\ast}(\mathcal{H})$ , $ C\in {GL^{+}(\mathcal{H})}$. A mapping F :$\Omega\rightarrow \mathcal{H}$ is called a continuous C-controlled K-Frame in $\mathcal{H}$ if :  
	\begin{itemize}
		\item F is weakly-measurable, ie,  for any $f\in \mathcal{H}$, the map
		\\
		$ w\rightarrow \langle f,F(w)\rangle_{\mathcal{A}}$ is  measurable on ${\Omega}$.\\
		\item There exists two strictly positive constants  $A$ and $B$  such that
		\begin{equation} \label{7}
		A\langle C^{\frac{1}{2}} K^{\ast}f, C^{\frac{1}{2}}K^{\ast}f\rangle_{\mathcal{A}} \leq\int_{\Omega}\langle f,F(w)\rangle_{\mathcal{A}} \langle CF(w),f\rangle_{\mathcal{A}} d{\mu(w)}\leq B\langle f,f\rangle_{\mathcal{A}} , f\in \mathcal{H}.
		\end{equation}
	\end{itemize}
	The elements $A$ and $B$ are called continuous C-controlled K-frame bounds. 
	
	If $A=B$ we call this continuous C-controlled K-Frame a continuous tight C-Controlled K-Frame, and if $A=B=1$ it is called a continuous Parseval C-Controlled K-Frame. If only the right-hand inequality of \eqref{7} is satisfied, we call  F a continuous C-controlled bessel mapping with Bessel bound $B$.
	
\end{definition}

\begin{example}
\begin{align*}
	H &=\mathcal{A}=l^{2}(\mathbb{C}) \\
	&=\left\{ \left\{ a_{n}\right\} _{n=1}^{\infty }\subset 
	\mathbb{C}
	%EndExpansion
	\text{ / }\sum\limits_{n=1}^{\infty }\left\vert a_{n}\right\vert
	^{2}<+\infty \right\} .
\end{align*}

$\mathcal{A}$ is recognized as a Hilbert $\mathcal{A}$-Module with the $\mathcal{A}$-inner product
\begin{equation*}
<\left\{ a_{n}\right\} _{n=1}^{\infty },\left\{ b_{n}\right\} _{n=1}^{\infty
}>_{\mathcal{A}}=\left\{ a_{n}\overline{b_{n}}\right\} _{n=1}^{\infty }.
\end{equation*}

Consider now the borned linear operator%
\begin{equation*}
\begin{array}{cccc}
C: & H & \rightarrow  & H \\ 
& \left\{ a_{n}\right\} _{n=1}^{\infty } & \longmapsto  & \left\{ \alpha
a_{n}\right\} _{n=1}^{\infty }%
\end{array}%
\end{equation*}
where $\alpha \in 
\mathbb{R}
%EndExpansion
_{+}^{\ast }$. Then $C$ is positive invertible and%
\begin{equation*}
C^{-1}(\left\{ a_{n}\right\} _{n=1}^{\infty })=\left\{ \alpha
^{-1}a_{n}\right\} _{n=1}^{\infty }.
\end{equation*}

Let $(\Omega ,\mu )$ the measure space where $\Omega =\left[ 0,1\right] $
and $\mu $ is the lebesgue measure and let 
\begin{equation*}
\begin{array}{cccc}
F: & \Omega  & \rightarrow  & H \\ 
& w & \longmapsto  & F_{w}=\left\{ \frac{w}{n}\right\} _{n=1}^{\infty }%
\end{array}%
.
\end{equation*}

In the author hand, consider the projection%
\begin{equation*}
\begin{array}{cccc}
K: & H & \rightarrow  & H \\ 
& \left\{ a_{n}\right\} _{n=1}^{\infty } & \longmapsto  & 
(a_{1},..,a_{r},0,...)%
\end{array}%
\end{equation*}

where $r$ is an integer $(r\geq 2)$.\\It's clair that $\ K^{\ast }=K$ and for
each $f=$ $\left\{ a_{n}\right\} _{n=1}^{\infty }\in H=l^{2}(%
\mathbb{C})$, one has%
\begin{eqnarray*}
	\int_{\Omega } &<&f,F_{w}>_{\mathcal{A}}<CF_{w},f>_{\mathcal{A}}d\mu
	(w)=\int_{\left[ 0,1\right] }\left\{ \frac{w}{n}a_{n}\right\} _{n=1}^{\infty
	}.\left\{ \alpha \frac{w}{n}\overline{a_{n}}\right\} _{n=1}^{\infty }d\mu (w)
	\\
	&=&\int_{\left[ 0,1\right] }\left\{ \alpha \frac{w^{2}}{n^{2}}\left\vert
	a_{n}\right\vert ^{2}\right\} _{n=1}^{\infty }d\mu (w) \\
	&=&\frac{\alpha }{3}\left\{ \frac{\left\vert a_{n}\right\vert ^{2}}{n^{2}}%
	\right\} _{n=1}^{\infty }.
\end{eqnarray*}

Hence%
\begin{equation*}
\int_{\Omega }<f,F_{w}>_{\mathcal{A}}<CF_{w},f>_{\mathcal{A}}d\mu (w)\leq 
\frac{\alpha \pi ^{2}}{18}<\left\{ a_{n}\right\} _{n=1}^{\infty },\left\{
a_{n}\right\} _{n=1}^{\infty }>_{\mathcal{A}}\text{.}
\end{equation*}

Furthermore,%
\begin{eqnarray*}
	&<&CK^{\ast }f,K^{\ast }f>_{\mathcal{A}}=<(\alpha a_{1},..,\alpha
	a_{r},0,...),(a_{1},..,a_{r},0,...)>_{\mathcal{A}} \\
	&=&(\alpha \left\vert a_{1}\right\vert ^{2},..,\alpha \left\vert
	a_{r}\right\vert ^{2},0,...).
\end{eqnarray*}

Then for $A=\frac{1}{3r^{2}},$ one obtain%
\begin{equation*}
\frac{\alpha }{3r^{2}}(\left\vert a_{1}\right\vert ^{2},..,\left\vert
a_{r}\right\vert ^{2},0,...)\leq \left\{ \frac{\alpha }{3}\frac{\left\vert
	a_{n}\right\vert ^{2}}{n^{2}}\right\} _{n=1}^{\infty }.
\end{equation*}

The conclusion is%
\begin{equation*}
\frac{1}{3r^{2}}<C^{1/2}K^{\ast }f,C^{1/2}K^{\ast }f>_{\mathcal{A}}\leq \int_{\Omega
}<f,F_{w}>_{\mathcal{A}}<CF_{w},f>_{\mathcal{A}}d\mu (w)\leq \frac{\alpha
	\pi ^{2}}{18}<f,f>_{\mathcal{A}}
\end{equation*}

\bigskip

\bigskip

\end{example}

    Let F be a  continuous C-controlled bessel mapping for Hilbert $C^{\ast}$- module $\mathcal{H}$ over $\mathcal{A}$.\\
    We define the  operator frame \\
    
   $S_C : \mathcal{H}  \rightarrow \mathcal{H}$ by, 
   $$S_Cf=\int_{\Omega}\langle f, F(\omega)\rangle_{\mathcal{A}} CF(\omega) d\mu(\omega).$$
   \begin{remark}
    From definition of $S$ and $S_C$, we have, $S_C=CS$.\\
   Using \cite{28} ,  $ S_{C}$ is $\mathcal{A}$-linear and bounded. Thus, it is adjointable.\\
   Since $\langle S_{C}x,x\rangle_{\mathcal{A}} \geq 0$, for any $x\in \mathcal{H}$, it result, again from \cite{28}, that  $ S_{C}$ is positive and selfadjoint.
   \end{remark}

\begin{theorem}\label{555}
	Let $\mathcal{\mathcal{H}}$ be a Hilbert $\mathcal{A}$-module,
	$K\in End_{\mathcal{A}}^{\ast}(\mathcal{H})$,  and $C\in {GL^{+}(\mathcal{H})}$. Let F : $\Omega\rightarrow \mathcal{H}$ a map. Suppose that $CK=KC$, $R(C^{\frac{1}{2}})\subset R(K^{\ast}C^{\frac{1}{2}})$ with $R(K^{\ast}C^{\frac{1}{2}})$ is closed.	Then F is a continuous C-controlled K-frame for $\mathcal{H}$ if and only if there exist two constants $0< A, B<{\infty}$ such that :
	\begin{equation}\label{6}
	A\| C^{\frac{1}{2}} K^{\ast}f\|^{2}\leq\|\int_{\Omega}\langle f,F(w)\rangle_{\mathcal{A}} \langle CF(w),f\rangle_{\mathcal{A}} d\mu(w)\|\leq B\|f\|^{2},  {f\in \mathcal{H}}.
	\end{equation}
	
\end{theorem}

\begin{proof}
	($\Longrightarrow$) obvious.\\
	For the converse, we suppose that 
	$0< A,  B<{\infty}$ such that :
	$$A\| C^{\frac{1}{2}} K^{\ast}f\|^{2}\leq\|\int_{\Omega}\langle f,F(w)\rangle_{\mathcal{A}} \langle CF(w),f\rangle_{\mathcal{A}} d\mu(w)\|\leq B\|f\|^{2},  {f\in \mathcal{H}}.$$
	We have, 
	
	\begin{align*}
	\|\int_{\Omega}\langle f,F(w)\rangle_{\mathcal{A}} \langle CF(w),f\rangle_{\mathcal{A}} d\mu(w)\|&=\|  \langle S_Cf,f\rangle_{\mathcal{A}}\|\\
	&=\|  \langle CSf,f\rangle_{\mathcal{A}} \|\\
	&=\|  \langle (CS)^{\frac{1}{2}}f,(CS)^{\frac{1}{2}}f\rangle_{\mathcal{A}} \|\\
	&=\|   (CS)^{\frac{1}{2}}f \|^{2}.
	\end{align*}

	Since,    $R(C^{\frac{1}{2}}) \subset R(K^{\ast}C^{\frac{1}{2}})  $ with
	$R(K^{\ast}C^{\frac{1}{2}})$ is closed, then by  theorem \ref{3}, there exists $0 \leq m$ such that,
	$$(C^{\frac{1}{2}})(C^{\frac{1}{2}})^{\ast}\leq m (K^{\ast}C^{\frac{1}{2}})(K^{\ast}C^{\frac{1}{2}})^{\ast}.$$
	 Thus, 
	 $$\langle (C^{\frac{1}{2}})(C^{\frac{1}{2}})^{\ast}f, f\rangle_{\mathcal{A}}\leq m \langle (K^{\ast}C^{\frac{1}{2}})(K^{\ast}C^{\frac{1}{2}})^{\ast}f,f\rangle_{\mathcal{A}}  .$$
	 Consequently,
	$$\|C^{\frac{1}{2}}f \|^{2}\leq m\|K^{\ast} C^{\frac{1}{2}}f \|^{2}.$$
	Then, $$A\|C^{\frac{1}{2}}f \|^{2}\leq A m\|K^{\ast} C^{\frac{1}{2}}f \|^{2}\leq m\|(CS)^{\frac{1}{2}}f \|^{2} .$$
	Hence,
	 $$\frac{A}{m}\|C^{\frac{1}{2}}f \|^{2}\leq\|(CS)^{\frac{1}{2}}f \|^{2}.$$
	So, 
	\begin{equation}\label{8}
		\sqrt\frac{A}{m}\|C^{\frac{1}{2}}f \|\leq\|(CS)^{\frac{1}{2}}f \|.	
	\end{equation}

	From lemma\ref{1}, we have,$$\sqrt\frac{A}{m}\langle C^{\frac{1}{2}}f,C^{\frac{1}{2}}f\rangle_{\mathcal{A}}\leq\langle C^{\frac{1}{2}}S^{\frac{1}{2}}f,C^{\frac{1}{2}}S^{\frac{1}{2}}f\rangle_{\mathcal{A}}.$$
	Then, $$\langle C^{\frac{1}{2}}f,C^{\frac{1}{2}}f\rangle_{\mathcal{A}}\leq\sqrt\frac{m}{A}\langle CSf,f\rangle_{\mathcal{A}}.$$
	So, $$\langle C^{\frac{1}{2}}f,C^{\frac{1}{2}}f\rangle_{\mathcal{A}}\leq\sqrt\frac{m}{A}\langle S_Cf,f\rangle_{\mathcal{A}} .$$
	One the deduce
	 $$\langle C^{\frac{1}{2}}K^{\ast}f,C^{\frac{1}{2}}K^{\ast}f\rangle_{\mathcal{A}}\leq\|K^{\ast}\|^2\langle C^{\frac{1}{2}}f,C^{\frac{1}{2}}f\rangle_{\mathcal{A}}\leq\|K^{\ast}\|^2\sqrt\frac{m}{A}\langle S_Cf,f\rangle_{\mathcal{A}}.$$
	 Hence,
	 \begin{equation}\label{1a}
	 	\frac{1}{\|K^{\ast}\|^2}\sqrt\frac{A}{m}\langle C^{\frac{1}{2}}K^{\ast}f,C^{\frac{1}{2}}K^{\ast}f\rangle_{\mathcal{A}}\leq\langle S_Cf,f\rangle_{\mathcal{A}}.
	 \end{equation}
	  
	Since $S_C$ is positive, selfadjoint and bounded $\mathcal{A}$-linear map, we can write $$\langle S_C^{\frac{1}{2}}f,S_C^{\frac{1}{2}}f\rangle_{\mathcal{A}}=\langle S_Cf,f\rangle_{\mathcal{A}}=\int_\omega\langle f,F(w)\rangle_{\mathcal{A}} \langle CF(w),f\rangle_{\mathcal{A}} d\mu(w). $$
	
	From lemma \ref{2}, there exists $D>0$ such that,
	 $$\langle S_C^{\frac{1}{2}}f,S_C^{\frac{1}{2}}f\rangle_{\mathcal{A}}\leq D \langle f,f\rangle_{\mathcal{A}},$$
	 
	hence,
	\begin{equation}\label{1b}
		\langle S_Cf,f\rangle_{\mathcal{A}}\leq D\langle f,f\rangle_{\mathcal{A}}. 
	\end{equation}
	 
	Therfore by \eqref{1a} and \eqref{1b}, we conclude that F is a continuous $C$-controlled $K$-frame in Hilbert $C^{\ast}$-module $\mathcal{H}$ with frame bounds $\frac{1}{\|K^{\ast}\|^2}\sqrt\frac{A}{m}$ and $D$.\\
	
\end{proof}
\begin{lemma}\label{111}
	Let $C\in GL^{+}(\mathcal{H})$. Suppose  $CS_C=S_CC$ and 
	$R(S_C^{\frac{1}{2}}) \subset R((CS_C)^{\frac{1}{2}})$ with 	$R((CS_C)^{\frac{1}{2}})$ is closed. Then $\|S_{C}^{\frac{1}{2}}f\|^{2} \leq \lambda \|(CS_{C})^{\frac{1}{2}}f\|^{2}$ for some $\lambda \geq0$.
\end{lemma}
   \begin{proof} 
   	By theorem\ref{3}, there exists some $\lambda >0$ such that,
   	$$(S_C^{\frac{1}{2}})(S_C^{\frac{1}{2}})^{\ast}\leq \lambda (CS_C^{\frac{1}{2}})(CS_C^{\frac{1}{2}})^{\ast}.$$
   	Hence, $$\langle (S_C^{\frac{1}{2}})(S_C^{\frac{1}{2}})^{\ast}f,f\rangle_{\mathcal{A}} \leq \lambda \langle (CS_C^{\frac{1}{2}})(CS_C^{\frac{1}{2}})^{\ast}f, f\rangle_{\mathcal{A}}. $$
   	So, $$\| S_C^{\frac{1}{2}}f\|^2\leq  \lambda \|(CS_C^{\frac{1}{2}})
   	f\|^2, f \in {\mathcal H}.$$
   \end{proof}

\begin{theorem}\label{11}
Let  F : $\Omega\rightarrow \mathcal{H}$ a map  and $C\in GL^{+}(\mathcal{H})$. Suppose $CS_C=S_CC$ and 
$R(S_C^{\frac{1}{2}}) \subset R((CS_C)^{\frac{1}{2}})$ with 	$R((CS_C)^{\frac{1}{2}})$ is closed. Then F is a continuous 
C-controlled Bessel mapping with bound B if and only if $U:l^2({\mathcal{A}})\rightarrow \mathcal{H}$ defined by 	$U(\{a_{w}\}_{w\in \Omega})=\int_{\Omega}a_w C F(w) d\mu(w)$ is well defined bounded with $\|U\|\leq\sqrt{B}\|C^\frac{1}{2}\|$.
	
\end{theorem}

\begin{proof}
	Assume that F is a continuous 	C-controlled Bessel with bound B.
	Hence , $$\|\int_{\Omega}\langle f,F(w)\rangle_{\mathcal{A}} \langle CF(w),f\rangle_{\mathcal{A}} d\mu(w)\|\leq B\|f\|^2, f\in {\mathcal{H}} . $$
	So, $$\|\langle S_Cf,f\rangle_{\mathcal{A}} \|\leq B\|f\|^2.$$
	In the begining, we show that U is  well defined . \\
	For each $\{a_w\}_{w{\in \Omega}} \in l^2(\mathcal{A})$,
	\begin{align*}
	\|U(\{a_w\}_{\omega\in {\Omega}})\|^2&= \underset{f\in \mathcal{H}, \|f\|=1} {\sup} {\|\langle U(\{a_w\}_{\omega\in {\Omega}}),f\rangle_{\mathcal{A}}\|^2 }\\&= \underset{f\in \mathcal{H}, \|f\|=1} {\sup}{ \|\langle \int_{\Omega} a_w CF(w)d\mu(w),f\rangle_{\mathcal{A}} \|^2} \\
	&=\underset{f\in \mathcal{H}, \|f\|=1} {\sup} \| \int_{\Omega} a_w \langle CF(w),f\rangle_{\mathcal{A}} d\mu(w)\|^2 \\    
	& \leq \underset{f\in \mathcal{H}, \|f\|=1} {\sup} \| \int_{\Omega}\langle f, C F(w)\rangle_{\mathcal{A}} \langle C F(w),f \rangle_{\mathcal{A}} d\mu(w)\| . 
	\| \int_{\Omega} a_wa_w^{\ast}d\mu(w)\|\\
	&=\underset{f\in \mathcal{H}, \|f\|=1} {\sup}\|\langle \int_{\Omega}\langle f, C F(w)\rangle_{\mathcal{A}}   C F(w)d\mu(w),f \rangle_{\mathcal{A}} \|. 
	\| \int_{\Omega} a_wa_w^{\ast}d\mu(w)\|\\
	&=\underset{f\in \mathcal{H}, \|f\|=1} {\sup} \|\langle CS_Cf,f \rangle_{\mathcal{A}}\|. 
	\| \int_{\Omega} a_wa_w^{\ast}d\mu(w)\|\\
	&=\underset{f\in \mathcal{H}, \|f\|=1} {\sup} \|\langle (CS_C)^{\frac{1}{2}}f,(CS_C)^{\frac{1}{2}}f \rangle_{\mathcal{A}}\|.
	\|\{a_w\}_{\omega\in {\Omega}} \|^2\\
	& \leq \underset{f\in \mathcal{H}, \|f\|=1} {\sup} \|(C)^{\frac{1}{2}}\|^2  \|(S_Cf)^{\frac{1}{2}} \|^2  \|\{a_w\}_{\omega\in {\Omega}} \|^2\\ 
	&\leq B \|(C)^{\frac{1}{2}}\|^2 \|\{a_w\}_{\omega\in {\Omega}} \|^2 . 
	\end{align*}
	
	Then, $$\| U \| \leq \sqrt{B} \|(C)^{\frac{1}{2}}\|. $$
	Hence $U$ is well defined and bounded.\\
	Now, suppose that $U$ is well defined, and  $$\| U \| \leq \sqrt{B} \|(C)^{\frac{1}{2}}\|. $$
	For any $f\in \mathcal{H}$ and $\{a_w\}_{\omega\in {\Omega}}\in l^2(\mathcal{A})$, we have, 
	\begin{align*}
	\langle f, U(\{a_w\}_{\omega\in {\Omega}})\rangle_{\mathcal{A}}&= \langle f,\int_{\Omega} a_w CF(w)d\mu(w)\rangle_{\mathcal{A}} \\ &=\int_\Omega \langle a_w^{\ast} C f, F(w)\rangle_{\mathcal{A}} d\mu(w) \\
	&=\int_\Omega \langle  C f, F(w)\rangle_{\mathcal{A}} a_w^{\ast} d\mu(w) \\
	&= \langle \{\langle Cf, F(w)\rangle_{\mathcal{A}} \}_{\omega\in {\Omega}} ,\{a_w\}_{\omega\in {\Omega}} \rangle_{\mathcal{A}}. 
	\end{align*}
	Then, $U$ has an adjoint, and $$U^{\ast}f=\{\langle Cf, F(w)\rangle_{\mathcal{A}} \}_{\omega\in {\Omega}}.$$
	Also,
\begin{align*}
\| U \|^{2}&= \underset{\|(\{a_w\}_{\omega\in {\Omega}})\|=1} {\sup} {\|U(\{a_w\}_{\omega\in {\Omega}})\|^2 }\\
&=\underset{\|(\{a_w\}_{\omega\in {\Omega}})\|=1, \|f\|=1} {\sup} {\|\langle U(\{a_w\}_{\omega\in {\Omega}}),f\rangle_{\mathcal{A}} \|^2 }\\
&=\underset{\|(\{a_w\}_{\omega\in {\Omega}})\|=1, \|f\|=1} {\sup} {\|\langle \{a_w\}_{\omega\in {\Omega}},U^{\ast}f\rangle_{\mathcal{A}} \|^2 }\\
&=\underset{\|f\|=1} {\sup} {\|U^{\ast}f \|^2 }\\
&=\|U^{\ast}\|^{2}
\end{align*} 
So,
	\begin{equation*}
	\|U^{\ast}f \|^2= \|\langle U^{\ast}f, U^{\ast}f \rangle_{\mathcal{A}}\| =\|\langle UU^{\ast}f, f \rangle_{\mathcal{A}}\|= \|\langle CS_Cf, f \rangle_{\mathcal{A}}\|.
	\end{equation*}
	Then,
	\begin{equation}\label{101}
	\|U^{\ast}f \|^2=\|(CS_C)^{\frac{1}{2}}f\|^2 \leq B\|(C)^{\frac{1}{2}}\|^2\|f \|^2. 
	\end{equation}
	From lemma \ref{111}, we have, $$\|(S_C)^{\frac{1}{2}}f\|^2 \leq \lambda\|(CS_C)^{\frac{1}{2}}f\|^2 ,$$
	for some $\lambda > 0$.\\
	Using \eqref{101} we get, 
	\begin{align*}
	\|(S_C)^{\frac{1}{2}}f\|^2 &\leq \lambda\|(CS_C)^{\frac{1}{2}}f\|^2 \\&\leq \lambda B\|C^{\frac{1}{2}}\|^2\|f\|^2.
	\end{align*}
	Hence F is a continuous C-controlled Bessel mapping with Bessel bound $\lambda B\|C^{\frac{1}{2}}\|^2$.
\end{proof}
	
\begin{proposition}\label{100}
	Let F be a continuous C-controlled K-frame for $\mathcal{H}$ with bounds A and B. Then :$$ACKK^{\ast} I\leq S_C \leq B.I.$$
\end{proposition}
\begin{proof}
	Suppose F is a continuous C-controlled K-frame with bounds A and B. Then, 
	$$ A \langle C^{\frac{1}{2}}K^{\ast}f,C^{\frac{1}{2}}K^{\ast}f\rangle_{\mathcal{A}} \leq \int_{\Omega}\langle f, F(w) \rangle_{\mathcal{A}} \langle C F(w), f \rangle_{\mathcal{A}} d\mu(w)\leq B\langle f , f \rangle_{\mathcal{A}} .$$
	Hence, $$ A \langle C KK^{\ast}f,f\rangle_{\mathcal{A}} \leq \langle S_Cf , f \rangle_{\mathcal{A}}\leq B\langle f , f \rangle_{\mathcal{A}}.   $$
	So, $$ACKK^{\ast}I \leq S_C \leq B.I.$$
\end{proof}
\begin{proposition}\label{8}
	Let F be a continuous C-controlled Bessel mapping for $\mathcal{H}$, and $C\in GL^{+}(\mathcal{H})  $. Then F is a continuous C-controlled K-frame for $\mathcal{H}$ if and only if there exists $A> 0$ such that: $$ACKK^{\ast} \leq CS. $$
\end{proposition}
\begin{proof}
	$(\Longrightarrow) $ obvious.\\
	$(\Longleftarrow)$ Assume that there exists $A> 0$ such that: $ACKK^{\ast} \leq CS $,\\
	then, $$ A \langle C KK^{\ast}f,f\rangle_{\mathcal{A}} \leq \langle S_Cf , f \rangle_{\mathcal{A}}.$$
	Hence,
	$$ A \langle C^{\frac{1}{2}}K^{\ast}f,C^{\frac{1}{2}}K^{\ast}f\rangle_{\mathcal{A}} \leq \langle S_Cf , f \rangle_{\mathcal{A}}.$$
	Therefore, 	
	$$ A \langle C^{\frac{1}{2}}K^{\ast}f,C^{\frac{1}{2}}K^{\ast}f\rangle_{\mathcal{A}} \leq 
	\int_{\Omega}\langle f, F(w) \rangle_{\mathcal{A}} \langle C F(w), f \rangle_{\mathcal{A}} d\mu(w).$$
	Hence F is a continuous C-controlled K-frame.\\
\end{proof}
\begin{proposition}\label{9}
    Let $ C \in GL^{+}(\mathcal{H})$, $ K \in End^{\ast}_{\mathcal{A}}(\mathcal{H})$ and F be a continuous C-controlled K-frame for $\mathcal{H}$ with lower and upper frames bounds A and B respectivelty. Suppose $KC=CK$ and  $R(C^{\frac{1}{2}}) \subset R(K^{\ast}C^{\frac{1}{2}}) $with 	$R(K^{\ast}C^{\frac{1}{2}}) $ is closed. Then  F is  continuous  K-frame for $\mathcal{H}$ with lower and upper frames bounds $A\| C^{\frac{-1}{2}}\|^{-2}\|(C)^{\frac{1}{2}}\|^{-2}$ and $B\|C^\frac{-1}{2}\|^{2}$ respectively.

\end{proposition}
\begin{proof}
    Assume that F is a continuous C-controlled K-frame with lower and upper frames bounds A and B. From theorem \ref{555}, we have:
    $$A\| C^{\frac{1}{2}} K^{\ast}f\|^{2} \leq\|\int_{\Omega}\langle f, F(w) \rangle_{\mathcal{A}} \langle C F(w), f \rangle_{\mathcal{A}} d\mu(w)\|\leq B\|f\|^2  , f\in \mathcal{H}. $$
    Then, 
    \begin{align*}
     A\|  K^{\ast}f\|^{2}&=A\| C^{\frac{-1}{2}} C^{\frac{1}{2}} K^{\ast}f\|^{2}\\
     &\leq A\| C^{\frac{-1}{2}}\|^2 \|C^{\frac{1}{2}} K^{\ast}f\|^{2}\\
     &\leq \| C^{\frac{-1}{2}}\|^2 \|\int_{\Omega}\langle f, F(w) \rangle_{\mathcal{A}} \langle C F(w), f \rangle_{\mathcal{A}} d\mu(w)\|.
    \end{align*}
    So,
    \begin{equation}\label{33}
     A\|  K^{\ast}f\|^{2}\leq\| C^{\frac{1}{2}}\|^2\|\langle S_Cf, f \rangle_{\mathcal{A}}\|.
    \end{equation}
    Moreover, 
    \begin{align*}
    \langle S_Cf, f \rangle_{\mathcal{A}}&=\langle CSf, f \rangle_{\mathcal{A}}\\
    &=\langle (CS)^{\frac{1}{2}}f, (CS)^{\frac{1}{2}}f \rangle_{\mathcal{A}}\\
    & =\|(CS)^{\frac{1}{2}}f\|^2\\
    &\leq \|(C)^{\frac{1}{2}}\|^2.\|(S)^{\frac{1}{2}}f\|^2\\
    &= \|(C)^{\frac{1}{2}}\|^2.\langle (S)^{\frac{1}{2}}f, (S)^{\frac{1}{2}}f \rangle_{\mathcal{A}}\\
    &= \|(C)^{\frac{1}{2}}\|^2.\langle Sf, f \rangle_{\mathcal{A}},
    \end{align*}
    then, 
    \begin{equation}\label{34}
  \langle S_Cf, f \rangle_{\mathcal{A}}\leq \|(C)^{\frac{1}{2}}\|^2.\langle Sf, f \rangle_{\mathcal{A}}	.
    \end{equation}
   From \eqref{33}and \eqref{34}, we have,
   \begin{align*}
    A\|K^{\ast}f\|^{2}&\leq\| C^{\frac{-1}{2}}\|^2\|(C)^{\frac{1}{2}}\|^2\langle Sf, f \rangle_{\mathcal{A}}\\&=\| C^{\frac{-1}{2}}\|^2\|(C)^{\frac{1}{2}}\|^2\int_{\Omega}\langle f, F(w) \rangle_{\mathcal{A}} \langle F(w), f \rangle_{\mathcal{A}} d\mu(w).
   \end{align*}
   Hence, 
   $$\| C^{\frac{-1}{2}}\|^{-2}\|(C)^{\frac{1}{2}}\|^{-2} A\|K^{\ast}f\|^{2}\leq\int_{\Omega}\langle f, F(w) \rangle_{\mathcal{A}} \langle F(w), f \rangle_{\mathcal{A}} d\mu(w). $$
   Moreover,
   \begin{align*}
    \|\int_{\Omega}\langle f, F(w) \rangle_{\mathcal{A}} \langle  F(w), f \rangle_{\mathcal{A}} d\mu(w)\|&=\| \langle Sf,f\rangle_{\mathcal{A}}\| \\
    &=\| \langle C^{-1}CSf,f\rangle_{\mathcal{A}}\|\\
    &=\| \langle (C^{-1}CS)^{\frac{-1}{2}}f,(C^{-1}CS)^{\frac{-1}{2}}f\rangle_{\mathcal{A}}\|\\
    &=\| (C^{-1}CS)^{\frac{1}{2}}f \|^2\\
    &\leq \| C^{\frac{-1}{2}} \|^2\|(CS)^{\frac{1}{2}}f \|^2\\ 
    &=\| C^{\frac{-1}{2}}\|^2\langle (CS)^{\frac{1}{2}}f,(CS)^{\frac{1}{2}}f\rangle_{\mathcal{A}}  \\
    &=\| C^{\frac{-1}{2}}\|^2 \langle CSf,f\rangle_{\mathcal{A}}\\
    &\leq \| C^{\frac{-1}{2}}\|^2 B \| f\|^2 .
   \end{align*}\\
   Then F is a continuous  K-frame for $\mathcal{H}$ with lower and upper frames bounds $A\| C^{\frac{-1}{2}}\|^{-2}\|(C)^{\frac{1}{2}}\|^{-2}$ and $B\|C^\frac{-1}{2}\|^{2}$.\\
\end{proof}
\begin{proposition}\label{10}
	Let $ C \in GL^{+}(\mathcal{H})$ and $ K \in End_{\mathcal{A}}^{\ast}(\mathcal{H})$. We Suppose that $KC=CK$, $R(C^{\frac{1}{2}}) \subset R(K^{\ast}C^{\frac{1}{2}}) $ with 	$R(K^{\ast}C^{\frac{1}{2}}) $ is closed and F is a continuous  K-frame for $\mathcal{H}$ with lower and upper frames bounds A and B respectivlty.\\
	Then  F is  continuous C-controlled K-frame for $\mathcal{H}$ with lower and upper frames bounds $A$ and   $\|C\|\|S\|$.
\end{proposition}

\begin{proof}
	Assume that F is a continuous  K-frame for $\mathcal{H}$ with lower and upper frames bounds A and B. Then we have:
	$$A\langle   K^{\ast}f, K^{\ast}f\rangle_{\mathcal{A}}\leq \int_{\Omega}\langle f, F(w) \rangle_{\mathcal{A}} \langle  F(w), f \rangle_{\mathcal{A}} d\mu(w)\leq B \langle f, f \rangle_{\mathcal{A}}, $$
	Since $\langle   K^{\ast}f, K^{\ast}f\rangle_{\mathcal{A}}>0$ and $\langle f, f \rangle_{\mathcal{A}}>0$ then, 
\begin{equation}
	A \| K^{\ast}f\|^2 \leq \|\int_{\Omega}\langle f, F(w) \rangle_{\mathcal{A}} \langle  F(w), f \rangle_{\mathcal{A}} d\mu(w)\| \leq B \|f\|^2.
\end{equation}
	Then for every $f\in \mathcal{H}$,
	\begin{align*}
	A\| C^{\frac{1}{2}} K^{\ast}f\|^{2}&=A\|K^{\ast}C^{\frac{1}{2}}f\|^{2}\\
	&\leq \|\int_{\Omega}\langle C^{\frac{1}{2}}f, F(w) \rangle_{\mathcal{A}} \langle F(w),C^{\frac{1}{2}} f \rangle_{\mathcal{A}} d\mu(w)\|\\
	&= \| \langle\int_{\Omega}  \langle C^{\frac{1}{2}}f, F(w) \rangle_{\mathcal{A}}  F(w)d\mu(w),C^{\frac{1}{2}} f \rangle_{\mathcal{A}} \|\\ &= \|\langle C^{\frac{1}{2}}Sf,C^{\frac{1}{2}}f\rangle_{\mathcal{A}} \|\\ 
	&=\|\langle CSf,f\rangle_{\mathcal{A}} \|\\
	&=\|\langle Sf,Cf\rangle_{\mathcal{A}} \|\\ &\leq  \|Sf\|. \|Cf\|,
	\end{align*}
	
	then
	\begin{equation}\label{201}
	A\| C^{\frac{1}{2}} K^{\ast}f\|^{2}	\leq\|\langle S_Cf,f\rangle_{\mathcal{A}} \| \leq\|S\|.\|C\|\|f\|^2.
	\end{equation}
	By \eqref{201}  and theorem\ref{555}, we conclude that F is  continuous C-controlled K-frame for $\mathcal{H}$ with lower and upper frames bounds $A$ and   $\|C\|\|S\|$.\\

\end{proof}
\begin{theorem}\label{11}
	Let $ C \in GL^{+}(\mathcal{H})$, and F be a  continuous C-controlled K-frame for $\mathcal{H}$ with bounds A and B. Let $ M,K \in End_{\mathcal{A}}^{\ast}(\mathcal{H})$ such that $R(M)\subset R(K)$, $R(K)$ is closed and C commutes with $M^{\ast}$ and $K^{\ast}$. Then F is  continuous C-controlled M-frame for $\mathcal{H}$.

\end{theorem}
   	
\begin{proof}
	Assume that F be a  continuous C-controlled K-frame for $\mathcal{H}$ with bounds A and B, then, 
	\begin{equation}\label{14}
	A\langle C^{\frac{1}{2}} K^{\ast}f,C^{\frac{1}{2}} K^{\ast}f\rangle_{\mathcal{A}}\leq\int_{\Omega}\langle f, F(w) \rangle_{\mathcal{A}} \langle CF(w), f \rangle_{\mathcal{A}} d\mu(w)\leq  B\langle f,  f \rangle_{\mathcal{A}}, f\in \mathcal{H} .
	\end{equation}
	
	Since $R(M)\subseteq R(K)$, by theorem \ref{3}, there exists some $0\leq \lambda$
	such that $$ M M^{\ast} \leq \lambda K K^{\ast} .$$
	Hence, $$\langle M M^{\ast}C^{\frac{1}{2}}f,C^{\frac{1}{2}}f\rangle_{\mathcal{A}} \leq \lambda \langle K K^{\ast}C^{\frac{1}{2}}f,C^{\frac{1}{2}}f\rangle_{\mathcal{A}} ,$$
	then, $$\frac{A}{\lambda} \langle M M^{\ast}C^{\frac{1}{2}}f,C^{\frac{1}{2}}f\rangle_{\mathcal{A}} \leq  A \langle K K^{\ast}C^{\frac{1}{2}}f,C^{\frac{1}{2}}f\rangle_{\mathcal{A}}. $$
	By \eqref{14}, we have,
    $$\frac{A}{\lambda} \langle  M^{\ast}C^{\frac{1}{2}}f, M^{\ast}C^{\frac{1}{2}}f\rangle_{\mathcal{A}}\leq \int_{\Omega}\langle f, F(w) \rangle_{\mathcal{A}} \langle C F(w), f \rangle_{\mathcal{A}} d\mu(w) \leq B \langle f, f\rangle_{\mathcal{A}}. $$
    Then F is  continuous C-controlled M-frame for $\mathcal{H}$ with bounds $\frac{A}{\lambda}$ and B.
\end{proof}  
    The following results gives the invariance of a continuous C-controlled Bessel mapping by a adjointable operator.\\
\begin{proposition}\label{15}
	Let $T\in End_{\mathcal{A}}^{\ast}(\mathcal{H})$ such that $TC=CT$ and F be a continuous C-controlled Bessel mapping with bound D. Then $TF$ is also a continuous C-controlled Bessel mapping with bound $D\|T^{\ast}\|$.\\
\end{proposition}	
\begin{proof}
	Assume that F is a continuous C-controlled Bessel mapping with bound D.
	Hence we have, $$\int_{\Omega}\langle f, F(w) \rangle_{\mathcal{A}} \langle C F(w), f \rangle_{\mathcal{A}} d\mu(w)\leq D \langle f,f\rangle_{\mathcal{A}},  f\in \mathcal{H}.$$
    We have,
    \begin{align*}
    \int_{\Omega}\langle f, TF(w) \rangle_{\mathcal{A}} \langle C TF(w), f \rangle_{\mathcal{A}} d\mu(w)&=\int_{\Omega}\langle T^{\ast}f, F(w) \rangle_{\mathcal{A}} \langle T C F(w), f \rangle_{\mathcal{A}} d\mu(w)\\
    &=\int_{\Omega}\langle T^{\ast} f, F(w) \rangle_{\mathcal{A}} \langle C F(w), T^{\ast}f \rangle_{\mathcal{A}} d\mu(w)\\ &\leq  D \langle T^{\ast}f,T^{\ast}f\rangle_{\mathcal{A}}\\ &\leq D\|T^{\ast}\|^2 \langle f,f\rangle_{\mathcal{A}}.
    \end{align*}
    	
    The result holds.
    
\end{proof}	
    Now, we study the invariance of a continuous C-controlled K-frame mapping by adjointable operator.\\
\begin{theorem}\label{16}
	Let $ C \in GL^{+}(\mathcal{H})$, and  F be a  continuous C-controlled K-frame for $\mathcal{H}$ with bounds A and B. If  $ T\in End_{\mathcal{A}}^{\ast}(\mathcal{H})$ with closed range such that   $R(K^{\ast}T^{\ast})$ is closed and C, K, T commute with each other. Then $TF$ is a continuous C-controlled K-frame for $R(T)$.
\end{theorem}
\begin{proof}
	Assume that F is a continuous C-controlled K-frame with bounds A and B. Then, 
	$$A\langle C^{\frac{1}{2}}K^{\ast}f,C^{\frac{1}{2}}K^{\ast}f\rangle_{\mathcal{A}}\leq \int_{\Omega}\langle f, F(w) \rangle_{\mathcal{A}} \langle C F(w), f \rangle_{\mathcal{A}} \leq B \langle f,f\rangle_{\mathcal{A}}, f\in \mathcal{H}. $$
	Since T has a closed range, then T has Moore-Penrose inverse $T^{\dagger}$ such that 
	$TT^{\dagger}T=T$ and $T^{\dagger}TT^{\dagger}=T^{\dagger}$, so $TT^{\dagger}_{/{ R(T)}}=I_{R(T)}$
	and 
	$(TT^{\dagger})^{\ast}=I^{\ast}=I=TT^{\dagger}$.\\
	We have,
	\begin{align*}
	\langle K^{\ast}C^{\frac{1}{2}}f,K^{\ast}C^{\frac{1}{2}}f\rangle_{\mathcal{A}}&=\langle (TT^{\dagger})^{\ast} K^{\ast}C^{\frac{1}{2}}f,(TT^{\dagger})^{\ast}K^{\ast}C^{\frac{1}{2}}f\rangle_{\mathcal{A}}\\
	&= \langle (T^{\dagger})^{\ast}T^{\ast} K^{\ast}C^{\frac{1}{2}}f,(T^{\dagger})^{\ast}T^{\ast}K^{\ast}C^{\frac{1}{2}}f\rangle_{\mathcal{A}}.
	\end{align*}
	So,
	\begin{equation}\label{301}
	\langle K^{\ast}C^{\frac{1}{2}}f,K^{\ast}C^{\frac{1}{2}}f\rangle_{\mathcal{A}}\leq \|(T^{\dagger})^{\ast}\|^2\langle T^{\ast} K^{\ast}C^{\frac{1}{2}}f,T^{\ast}K^{\ast}C^{\frac{1}{2}}f\rangle_{\mathcal{A}}.	
	\end{equation}
	Therfore,
	\begin{equation}\label{55}
	 \|(T^{\dagger})^{\ast}\|^{-2} \langle  K^{\ast}C^{\frac{1}{2}}f,K^{\ast}C^{\frac{1}{2}}f\rangle_{\mathcal{A}} \leq \langle T^{\ast} K^{\ast}C^{\frac{1}{2}}f,T^{\ast}K^{\ast}C^{\frac{1}{2}}f\rangle_{\mathcal{A}}.
	\end{equation}
	
	Consequently, from theorem \ref{3}, and $R(T^{\ast}K^{\ast}) \subset R(K^{\ast}T^{\ast})$, there exists some $\lambda \geq 0$ such that,
	
	\begin{equation}\label{18}
	\langle T^{\ast}K^{\ast}C^{\frac{1}{2}}f, T^{\ast}K^{\ast}C^{\frac{1}{2}}f\rangle_{\mathcal{A}} \leq \lambda \langle K^{\ast}T^{\ast}C^{\frac{1}{2}}f, K^{\ast}T^{\ast}C^{\frac{1}{2}}f\rangle_{\mathcal{A}}. 
	\end{equation}
	
	Hence, using \eqref{55} and \eqref{18} we have,
	\begin{align*}
	\int_{\Omega}\langle f, TF(w) \rangle_{\mathcal{A}} \langle C TF(w), f \rangle_{\mathcal{A}} d\mu(w)
	&=\int_{\Omega}\langle T^{\ast}f, F(w) \rangle_{\mathcal{A}} \langle T C F(w), f \rangle_{\mathcal{A}} d\mu(w)\\
	&=\int_{\Omega}\langle T^{\ast}f, F(w) \rangle_{\mathcal{A}} \langle  C F(w),T^{\ast} f \rangle_{\mathcal{A}} d\mu(w)\\
	&\geq A \langle C^{\frac{1}{2}} K^{\ast} T^{\ast} f,C^{\frac{1}{2}} K^{\ast} T^{\ast}f\rangle_{\mathcal{A}}\\ &\geq\frac{A}{\lambda} \langle T^{\ast} C^{\frac{1}{2}} K^{\ast}  f,T^{\ast}C^{\frac{1}{2}} K^{\ast} f\rangle_{\mathcal{A}},   
	\end{align*}
	then,
	\begin{equation}\label{66}
	\int_{\Omega}\langle f, TF(w) \rangle_{\mathcal{A}} \langle C TF(w), f \rangle_{\mathcal{A}} d\mu(w)\geq \frac{A}{\lambda} \|(T^{\dagger})^{\ast}\|^{-2} \langle  C^{\frac{1}{2}} K^{\ast}  f,C^{\frac{1}{2}} K^{\ast} f\rangle_{\mathcal{A}}	
	\end{equation}

	Using \eqref{66} and proposition\ref{15}, the result holds.\\

\end{proof}
\begin{theorem}\label{19}
	Let $ C \in GL^{\dagger}(\mathcal{H})$ and F be a  continuous C-controlled K-frame for $\mathcal{H}$ with bounds A and B. \\If $ T\in End_{\mathcal{A}}^{\ast}(\mathcal{H})$ is a isometry such that $R(T^{\ast}K^{\ast}) \subset R(K^{\ast}T^{\ast})$ with    $R(K^{\ast}T^{\ast})$ is closed and C, K, T commute with each other, then $TF$ 
	is a continuous C-controlled K-frame for $\mathcal{H}$.
	
\end{theorem}

\begin{proof}
	Using theorem \ref{3}, there exists some $\lambda\geq 0$ such that,
    $$\|T^{\ast}K^{\ast}C^{\frac{1}{2}}f\|^2\leq \lambda \|K^{\ast}T^{\ast}C^{\frac{1}{2}}f\|^2.$$
    Assume A the lower bound for the  continuous C-controlled K-frame F and T is an isometry then,
    \begin{align*}
     {\frac{A}{\lambda}}\|C^{\frac{1}{2}}K^{\ast}f\|^2&={\frac{A}{\lambda}}\|T^{\ast}C^{\frac{1}{2}}K^{\ast}f\|^2\\
    &\leq A \|K^{\ast}T^{\ast}C^{\frac{1}{2}}f\|^2 \\& = A \|C^{\frac{1}{2}}K^{\ast}T^{\ast}f \|^2\\ &\leq \int_{\Omega}\langle T^{\ast}f, F(w) \rangle_{\mathcal{A}} \langle C F(w), T^{\ast} f \rangle_{\mathcal{A}} d \mu(w)\\
    &=\int_{\Omega}\langle f, TF(w) \rangle_{\mathcal{A}} \langle TC F(w),  f \rangle_{\mathcal{A}} d \mu(w),
    \end{align*}
    then, 
    \begin{equation}\label{77}
    {\frac{A}{\lambda}}\|C^{\frac{1}{2}}K^{\ast}f\|^2\leq\int_{\Omega}\langle f, TF(w) \rangle_{\mathcal{A}} \langle CT F(w),  f \rangle_{\mathcal{A}} d \mu(w).  	
    \end{equation}
    
    Hence, from proposition \ref{15} and inequality \eqref{77}, we conclude that $TF$ is a continuous C-controlled K-frame for $\mathcal{H}$ with bounds ${\frac{A}{\lambda}}$ and $B \|T^{\ast}\|^2$.\\
    
\end{proof}

\bibliographystyle{amsplain}

\vspace{0.1in}

\end{document}